\documentclass[11pt]{amsart}
\usepackage{amsmath}
\usepackage{cases}
\usepackage{mathrsfs}
\usepackage{amssymb}
\usepackage{txfonts}
\usepackage{amscd}
\usepackage{braket}
\usepackage{amsfonts,latexsym,amsmath,amsthm,amsxtra,mathdots}
\usepackage[bookmarks,colorlinks]{hyperref}
\usepackage[all,cmtip]{xy}
\RequirePackage{amsmath} \RequirePackage{amssymb}
\usepackage{color}
\usepackage{mathtools}
\mathtoolsset{showonlyrefs}

     \renewcommand{\Im}{{\mathfrak{Im}\,}}

\def\-{^{-1}}

\def\'{^{\prime}}

\def\-{^{-1}}

\newcommand{\delete}[1]{}

    \theoremstyle{plain}
\usepackage{colordvi}
\usepackage{multicol}
\usepackage{verbatim}
\numberwithin{equation}{section}
\newtheorem{thm}{Theorem}[section]
\newtheorem{lem}{Lemma}[section]
\newtheorem{rem}{Remark}[section]
\newtheorem{de}{Definition}[section]
\newtheorem{pro}{Proposition}[section]

\def\({\left( }
\def\){\right) }

\begin{document}

\title{The sup-norm of holomorphic cusp forms}
\author{Zhilin Ye}
\maketitle

\begin{abstract}
Let $f$ be a normalized holomorphic cusp form with a square-free level $N$ and weight $k$. Using a pre-trace formula, we establish a sup-norm bound of $f$ such that $\|y^kf(z)\|_{\infty} \ll N^{-1/6+\epsilon}$ where the trivial bound is $\|y^kf(z)\|_{\infty} \ll 1$. This result is an analog of a similar bound in Maa{\ss} form case. 
\end{abstract}

\section{Introduction and Main Results}
The holomorphic cusp forms with weight $k$ and level $N$ are holomorphic functions on the upper half-plane $F : \mathbb{H}^2 \rightarrow \mathbb{C}$ satisfying
$$F(\gamma z) = (cz+d)^{k}F(z),$$
when 
$$\gamma = \left(\begin{array}{cc}a & b \\c & d\end{array}\right) \in \Gamma_0(M),$$
 and vanishing at every cusp. Denote by $\mathcal{S}_{k}(N)$ the space consisting of all such functions. Any element $f\in \mathcal{S}_{k}(M)$ has a Fourier series expansion at infinity
 $$f(z) = \sum_{n\geqslant 1}\frac{\psi_f(n)}{n^{\frac{1}{2}}}(n)^{\frac{k}{2}}e(nz)$$
 with coefficients $\psi_f(n)$ satisfying
 $$\psi_f(n) \ll_f \tau(n)$$
as proven by Deligne. In this paper, $e(z)$ always means $e^{2\pi i z}$.\\

We can choose an orthonormal basis $\mathcal{B}_k(N)$ of $\mathcal{S}_k(N)$ which consists of eigenfunctions of all the Hecke operators $T_n$ with $(n,N) =1$. If a cusp form $f$ is an eigenfunction of the Hecke operator $T_n$, we denote by $\lambda_f(n)$ the eigenvalue of $f$.\\

There is a subset $\mathcal{B}^{\star}_k(N)$ of $ \mathcal{B}_k(N)$ which consists of all the \emph{newforms}. It is well known that these forms are eigenfunctions of all the Hecke operators $T_m$ even for $(m,N)\neq1$. \\

Denote by $\braket{f, g} : = \int_{\mathbb{H}^2/\Gamma_0(N)}f\bar{g}y^{k-2}dxdy$ the Petersson inner product of two forms $f$ and $g$. Then we have the following bound. %Under the normalization $\|f\| =1$, our main theorem can be stated as following:\\
%Now, under the normalization $\psi_f(1)=1$, our main theorem can be stated as following:

\begin{thm}\text{(Sup-norm for holomorphic case)}\label{sup}
Let $f \in \mathcal{B}^{\star}_k(N)$ with square-free level $N$ and weight $k>2$. Then for any $\epsilon >0$ we have a bound
$$\|y^{\frac{k}{2}}f(z)\|_{\infty} \ll_{\epsilon} k^{\frac{1}{2}}N^{-\frac{1}{6}+\epsilon}\braket{f, f}^{1/2}.$$
\end{thm}

\begin{rem}
This result is first claimed in \cite{HT3}. But the author is not aware of any written proof.
\end{rem}

\begin{rem}
%Lemma \ref{sup} should be already obtained by N. Templier. But 
%The author is not aware of any written proof for the holomorphic case. But it can be derived through a similar argument with a more cas in \cite{HT3}. 
The trivial sup-norm bound is $N^{\frac{1}{2}}$ under our normalization. The first nontrivial bound is given by Blomer and Holowinsky in \cite{BRH}. Then, several improvements are made by Harcos and Templier in \cite{HT1}, \cite{HT2} and \cite{HT3}. Moreover, a hybrid bound is obtained by Templier in \cite{T}. 
\end{rem}

The proof follows the same lines as in \cite{HT3} and \cite{T}. 
 
 \section{Preliminaries} Let $N$ be a positive square-free integer.
 
 \subsection{The Sup-norm via Fourier Expansion}
We first need to establish a bound of $f$ when $y$ is large. 
 \begin{pro}\label{trivial-sup}
 \begin{align}
y^{k/2}f(z)\braket{f, f}^{-1/2}N^{1/2} \ll 
 \begin{cases}
k^{1/4+\epsilon}y^{-1/2} + y^{1/2}k^{\epsilon -1/4}, & \text{ if }y \ll k,\\
 k^{1/4+\epsilon}y^{-1/2} + 2^{k/2}k^{\epsilon}(2\pi y)^{k/2+\epsilon}e^{-2\pi y}\Gamma(k)^{-1/2}, & \text{ if } y\gg k.
 \end{cases}
 \end{align}
 \end{pro}
 
 \begin{rem}
 This proposition is implicitly proved in \cite{X}. 
 \end{rem}
 
 \subsection{Pretrace Formula for Holomorphic Cusp Forms}
Let
\begin{align}
h(z,w) : = \sum_{\gamma \in \Gamma_0(N)}\frac{1}{(j(\gamma,z))^k}\frac{1}{(w+\gamma.z)^k},
\end{align}
where $j(\gamma,z):=cz+d$ if $\gamma = \left(\begin{array}{cc}a & b \\c & d\end{array}\right)$.\\

We have a pre-trace formula as following. See \cite{RO} Appendix $1$ for the details.
\begin{lem}\label{PT}
Let $C_k =\frac{(-1)^{k/2}\pi}{2^{(k-3)}(k-1)}$. Then
$$C_k^{-1}h(z,w) =\sum_{i=1}^{J}\frac{f_i(z)\overline{f_i(-\overline{w})}}{\braket{f_i,f_i}},$$
where the sum is over an orthonormal basis of holomorphic cusp forms of weight $k$ and level $N$.
\end{lem}

Define Atkin-Lehner operators as following:
\begin{de}\label{Atkin-Lehner}
Atkin-Lehner operators of level $N$ are defined to be the elements in the set 
\begin{align}
A_0(N):=\left\{\sigma = \left(\begin{array}{cc}\sqrt{r}a & \frac{b}{\sqrt{r}} \\ \sqrt{r}s & \sqrt{r}d \end{array}\right) : \sigma \in SL_2(\mathbb{R}), r|N, N|rs, a, b, s, d \in \mathbb{Z},(a,s)=1\right\}.
\end{align}
\end{de}

A well known result is
\begin{lem}\label{A-Inv}
Let $f(z)$ be a holomorphic cusp newform of level $N$ and weight $k$. Then the function $F(z):=|y^{k/2}f(z)|$ is $A_0(N)$-invariant.
\end{lem}

 \subsection{Amplification Method}
 Let $T_l$ be Hecke operators as defined in \cite{HT3}. Choose a basis of modular forms which consists of Hecke eigenforms. Let 
 $$\Lambda = \left\{ p \in \mathbb{Z} : p \text{ prime }, (p,N)=1, L\leqslant p<2L\right\},$$ 
also let
$$\Lambda^2 = \left\{ p^2 : p\in \Lambda\right\}.$$

%Let 
%$$G_l(M) = \left\{\left(\begin{array}{cc}a & b \\c & d\end{array}\right) \in M_2(\mathbb{Z}): M|c, \det = l\right\}.$$\\

We define that
\begin{de}
Let 
$$G_l(N) : = \left\{\gamma = \left(\begin{array}{cc}a & b \\c & d\end{array}\right): a,b,c,d \in \mathbb{Z}, N|c, \det(\gamma) = l\right\}.$$
Let 
$$u_{\gamma}(z) := \frac{j(\gamma,z)(\overline{z}-\gamma.z)}{\text{Im}(z)}.$$
Let
%be the function of distance on $\mathbb{H}^2$. Let 
$$M(z,l,\delta) := \#\left\{\gamma \in G_l(N) : u(\gamma z, z) \leqslant \delta \right\}.$$
\end{de}

For any finite sequence of complex numbers $\{y_l\}$, we have
\begin{align}
\sum_{l}y_lT_l\(h(z,\cdot)\) & = \sum_{l}\frac{y_l}{\sqrt{l}}\sum_{\alpha \in G_l(N)} (\det \alpha)^{k/2}\frac{1}{j(\alpha,z)^k}\frac{1}{(\cdot+\alpha.z)^k}.
\end{align}
Otherwise, by Lemma \ref{PT}, we have 
\begin{align}
\sum_{l}y_lT_l\(h(z,\cdot)\)  = C_k\sum_{l}y_l\sum_{i=1}^{J}\frac{T_l\(f_i(z)\)\overline{f_i(-\overline{\cdot})}}{\braket{f_i,f_i}}= C_k\sum_{l}y_l\sum_{i=1}^{J}\frac{\lambda_i(l)f_i(z)\overline{f_i(-\overline{\cdot})}}{\braket{f_i,f_i}}.
\end{align}
Hence, by chosing $\cdot = -\overline{z}$, we have
\begin{align}\label{eq8.1}
C_k\sum_{i=1}^{J}\sum_{l}y_l\lambda_i(l)\frac{y^kf_i(z)\overline{f_i(z)}}{\braket{f_i,f_i}}  =  \sum_{l}\frac{y_l}{\sqrt{l}}\sum_{\alpha \in G_l(N)} (\det \alpha)^{k/2}\frac{y^k}{j(\alpha,z)^k}\frac{1}{(-\overline{z}+\alpha.z)^k} = \sum_{l}y_ll^{\frac{k-1}{2}}\sum_{\alpha \in G_l(N)} u_{\alpha}(z)^{-k}.
\end{align}

We then establish an "amplified" version of the formula above. By the multiplicity of the erigenvalues, for any sequence of complex numbers ${x_l}$, we get
\begin{align}\label{amp}
C_k\sum_{i=1}^{J}\left|\sum_{l}x_l\lambda_i(l)\right|^2\frac{|y^{k/2}f_i(z)|^2}{\braket{f_i,f_i}}&= C_k\sum_{i=1}^{J}\sum_{l_1,l_2}x_{l_1}\overline{x_{l_2}}\lambda_i(l_1)\overline{\lambda_i(l_2)}\frac{|y^{k/2}f_i(z)|^2}{\braket{f_i,f_i}}\\
& = C_k\sum_{i=1}^{J}\sum_{l}y_{l}\lambda_i(l)\frac{|y^{k/2}f_i(z)|^2}{\braket{f_i,f_i}}\\
& = \sum_{l}y_ll^{\frac{k-1}{2}}\sum_{\alpha \in G_l(N)} u_{\alpha}(z)^{-k},
\end{align}
where 
$$y_l : = \sum_{\substack{d|(l_1,l_2) \\ l = l_1l_2/d^2}}x_{l_1}\overline{x_{l_2}}.$$

Now, let 
$$x_l := \begin{cases} \text{sign}(\lambda_i(l)) & \text{ if } l \in \Lambda \cup \Lambda^2\\ 0 & \text{ otherwise }\end{cases}.$$
We therefore have
$$\left|\sum_{l}x_l\lambda_i(l)\right|\gg_{\epsilon} L^{1-\epsilon}.$$
Indeed, this follows from the relation $\lambda_{i}(l)^2 - \lambda_{i}(l^2) = 1$, which implies that $\max\left\{\left|\lambda_{i}(l)\right|, \left|\lambda_{i}(l^2)\right|\right\} \geqslant 1/2$.\\

As the way in \cite{HT3}, we split the counting of matrices $\gamma = \left(\begin{array}{cc}a & b \\c & d\end{array}\right)$ as
$$M = M_{*}+M_u+M_p$$
according to whether $c \neq 0$ and $(a+d)^2 \neq 4l$ (generic), or $c=0$ and $a\neq d$ (upper-triangular), or $(a+d)^2 = 4l$ (parabolic).\\

Moreover, we have 
\begin{lem}\label{lower}
If $\delta < 2\sqrt{l}$, $M(z,l,\delta) =0.$\\
\end{lem}

\begin{proof}
It suffices to show that $|u_{\gamma}(z)| \geqslant 2\sqrt{l}$ when $\gamma \in G_l(N)$. When $\mathrm{Trace}(\gamma) \geqslant 2\sqrt{l}$, we have $|u_{\gamma}(z)| \geqslant |\Im u_{\gamma}(z)| =  \mathrm{Trace}(\gamma) \geqslant 2\sqrt{l}$. When $\mathrm{Trace}(\gamma) < 2\sqrt{l}$, let $g\in SL_2(\mathbb{R})$ be a matrix such that 
$$g^{-1}\gamma g= \left(\begin{array}{cc}\sqrt{l}\cos \theta & \sqrt{l}\sin \theta \\-\sqrt{l}\sin \theta & \sqrt{l}\cos \theta\end{array}\right),$$
where $\theta \in \mathbb{R}$. By a direct calculation, we have $|u_{g^{-1}\gamma g}(z)| = |u_{\gamma}(gz)|$. Let $w = g^{-1}z = x+iy$, then
\begin{align}
|u_{\gamma}(z)|^2 = |u_{g^{-1}\gamma g}(w)|^2 = ly^{-2}|\sin^2\theta(1+|w|^2)^2 + 4y^2\cos^2\theta| \geqslant 4l.
\end{align}
\end{proof}

\begin{rem}
A calculation with full details can be found in \cite{RO} Appendix B.
\end{rem}

By \eqref{amp}, we have
\begin{align}\label{supineq}
C_kL^{2-\epsilon}\frac{\left|y^{k/2}f_i(z)\right|^2}{\braket{f_i,f_i}} & \ll \sum_{l}|y_l|l^{\frac{k-1}{2}}\sum_{\alpha \in G_l(N)} |u_{\alpha}(z)|^{-k}\\
&=\sum_{l}|y_l|l^{\frac{k-1}{2}}\sum_{\substack{\alpha \in G_l(N)\\ \alpha \text{ parabolic}}} |u_{\alpha}(z)|^{-k} +\sum_{l}|y_l|l^{\frac{k-1}{2}}\sum_{\substack{\alpha \in G_l(N)\\ \alpha \text{ generic or upper-triangular}}} |u_{\alpha}(z)|^{-k}\\
&\ll \sum_{l}|y_l|l^{\frac{k-1}{2}}\sum_{\substack{\alpha \in G_l(N)\\ \alpha \text{ parabolic}}} |u_{\alpha}(z)|^{-k} + \sum_{l}|y_l|l^{\frac{k-1}{2}}\int_{0}^{\infty}\delta^{-k}d \(M_u+M_*\)(z,l,\delta)\\
&\ll \sum_{l}|y_l|l^{\frac{k-1}{2}}\sum_{\substack{\alpha \in G_l(N)\\ \alpha \text{ parabolic}}} |u_{\alpha}(z)|^{-k} +k\sum_{l}|y_l|l^{\frac{k-1}{2}}\int_{2\sqrt{l}}^{\infty}\frac{\(M_u+M_*\)(z,l,\delta)}{\delta^{k+1}}d\delta,
\end{align}
where the last step follows from integration by parts and Lemma \ref{lower}.\\

The remaining problem is to establish an upper-bound for $M_*$, $M_u$ and the sum over parabolic matrices. %For the generic and upper-triangular matrices, we estimate $M_*$ and $M_u$ separately as in \cite{HT3}. For the parabolic case, we do brute force calculation.\\

\subsection{Counting Lattice Points}
As in \cite{HT3}, we estimate the sum of $M_*(z,l,\delta)$ and the sum of $M_u(z,l,\delta)$ separately.\\
%We split the counting of matrices $\gamma = \left(\begin{array}{cc}a & b \\c & d\end{array}\right)$ as
%$$M = M_{*}+M_u+M_p$$
%according to whether $c \neq 0$ and $(a+d)^2 \neq 4l$ (generic), or $c=0$ and $a\neq d$ (upper-triangular), or $(a+d)^2 = 4l$ (parabolic).\\

We state two lemmas in \cite{HT3} below.
\begin{lem}[\cite{HT3} Lemma 2.1]\label{Countlem1}
Let $\Theta$ be a eucilidean lattice of rank $2$ and $D$ be a disc of radius $R>0$ in $\Theta \otimes_{\mathbb{Z}}\mathbb{R}$ (not neceesarily centered at $0$). If $\lambda_1 \leqslant \lambda_2$ are the successive minima of $\Theta$, then
\begin{align}\label{CP}
\#(\Theta\cap D) \ll 1+\frac{R}{\lambda_1}+\frac{R^2}{\lambda_1\lambda_2}.
\end{align}
\end{lem}

\begin{lem}[\cite{HT2} Lemma 1]\label{Countlem2}
Let $z \in A_0(N) \backslash \mathbb{H}^2$. Then we have
\begin{align}\label{zregion}
\text{Im }z \geqslant \frac{\sqrt{3}}{2N}
\end{align}
and for any $(c,d)\in\mathbb{Z}^2$ distinct from $(0,0)$ we have 
\begin{align}\label{lattice}
|cz+d|^2\geqslant \frac{1}{N}.
\end{align}
\end{lem}

\begin{rem}
This is the where the square-free condition comes into play. \eqref{lattice} is not true when $N = q^2$ for an integer $q$. For example, let $z = \frac{1}{q}+i\frac{\sqrt{3}}{2q^2}$, then it is easy to check that $z$ is in the fundamental domain but the lattice generated by $(1,z)$ behaves badly.
\end{rem}

Then, we have  
\begin{lem}
For any $z = x+iy \in A_0(N) \backslash \mathbb{H}^2$ and $1 \leqslant \Lambda \leqslant N^{O(1)}$, $M_*(z,l,\delta)=0$ if $2\delta<Ny$. Moreover
\begin{align}\label{m*1}
\sum_{1\leqslant l \leqslant \Lambda}M_{*}(z,l,\delta) \ll \(\frac{\delta^2}{Ny}+\frac{\delta^3}{N^{1/2}}+\frac{\delta^4}{N}\)N^{\epsilon},
\end{align}
\begin{align}\label{m*2}
\sum_{\substack{1\leqslant l \leqslant \Lambda\\ l \text{ square}}}M_{*}(z,l,\delta) \ll \(\frac{\delta}{Ny}+\frac{\delta^2}{N^{1/2}}+\frac{\delta^3}{N}\)N^{\epsilon}.
\end{align}
For $1 \leqslant l_1 \leqslant \Lambda \leqslant N^{O(1)}$,
\begin{align}\label{m*3}
\sum_{1\leqslant l \leqslant \Lambda}M_{*}(z,l_1l^2,\delta) \ll \(\frac{\delta}{Ny}+\frac{\delta^2}{N^{1/2}}+\frac{\delta^3}{N}\)N^{\epsilon}.
\end{align}
\end{lem}

\begin{proof}
By the definition of $M_{*}$, we count the number of matrices $\alpha = \left(\begin{array}{cc}a & b \\c & d\end{array}\right)$ such that 
\begin{align}\label{or}
|u_{\alpha}(z)| = |az+b-\overline{z}(cz+d)|\frac{1}{y} =|l+|cz+d|^2 -(cz+d)(a+d)|\frac{1}{cy} \leqslant \delta.
\end{align}
By considering the imaginary part, we obtain
\begin{align}\label{a+d}
|a+d| \leqslant \delta.
\end{align}
By considering the real part, we obtain
\begin{align}
|l+|cz+d|^2 -(cx+d)(a+d)|\leqslant \delta|cy|.
\end{align}
We therefore have
\begin{align}
|l+|cz+d|^2|\leqslant \delta\(|cy|+|cx+d|\) \leqslant 2\delta|cz+d|.
\end{align}
Since $l>0$, we obtain that
\begin{align}
|cz+d|\leqslant 2\delta.
\end{align}
Furthermore, by the inequalities above, we get $|cy| \leqslant 2\delta$.\\

Otherwise, we have that $N|c$ and $c\neq 0$ in this case. Hence when $2\delta/y<N$, $M_* = 0$. This proves our first claim.\\

By \eqref{or}, 
\begin{align}
|az+b-\overline{z}(cz+d)| = |(a-d)z + b -cz^2 + (cz+d)(z-\overline{z})| \leqslant \delta y,
\end{align}
which implies that
\begin{align}\label{mainineq}
|(a-d)z+b -cz^2|\ll \delta y.
\end{align}

Consider the lattice $\braket{1,z}$ inside $\mathbb{C}$. Its covolume equals $y$. By \eqref{lattice}, the shortest distance between two different points in the lattice is at least $N^{-1/2}$. In \eqref{mainineq}, we are counting lattice points $(a-d,b)$ in a disc of volume $\ll \delta^2y^2$ centered at $cz^2$. Thus, by \eqref{CP}, there are $\ll 1+\frac{\delta y}{N^{-1/2}} + \frac{\delta^2y^2}{y}$ possible pairs $(a-d,b)$ for each $c$.\\

When $l$ is a general number, since $|a+d|\ll \delta$, we have $\ll \delta$ many possible $a+d$ for a given triple $(a-d,b,c)$.\\

Now, consider
\begin{align}\label{det}
(a-d)^2 + 4bc = (a+d)^2 -4l.
\end{align}
When $l$ is a square, for any given triple $(a-d,b,c)$, the number of pairs $(a+d,l)$ satisfying \eqref{det} is $\ll N^{\epsilon}$.\\

When $l = l_1l_2^2$ and $l_1$ is square-free, \eqref{det} becomes a Pell equation. So the solution is a power of fundamental unit which is always greater than $\frac{1+\sqrt{5}}{2}$. Therefore, the number of pairs $(a+d, l_2)$ satisfying \eqref{det} is $\ll N^{\epsilon}$ .\\

Finally, since $c \ll \delta /y$ and $N|c$, we have $\ll \delta /Ny$ possible values for $c$ for all these three cases above. For each $c$, we have  $\ll 1+\frac{\delta y}{N^{-1/2}} + \frac{\delta^2y^2}{y}$ possible pairs $(a-d,b)$. For each $(a-d,b,c)$, we have $\ll \delta$ possible $(a+d,l)$ for the case in \eqref{m*1}. And for the cases in \eqref{m*2} and \eqref{m*3}, we have $\ll N^{\epsilon}$ possible $(a+d,l)$. The proof is completed.
\end{proof}

\begin{lem}\label{CON}
For any $z=x+iy \in A_0(N) \backslash \mathbb{H}^2$ and $1\leqslant  \Lambda \leqslant N^{O(1)}$, the following estimations hold true when $l_1,l_2$ and $l_3$ runs over primes.
\begin{align}\label{mu0}
\sum_{1 \leqslant l_1\leqslant \Lambda}M_u(z,l_1,\delta) \ll \(1+\delta N^{1/2}y + \delta^2y\)N^{\epsilon},
\end{align}
\begin{align}\label{mu1}
\sum_{1 \leqslant l_1l_2 \leqslant \Lambda}M_u(z,l_1l_2,\delta)\ll \(\Lambda+\Lambda\delta N^{1/2}y + \Lambda\delta^2y\)N^{\epsilon},
\end{align}
\begin{align}\label{mu2}
 \sum_{1 \leqslant l_1l_2 \leqslant \Lambda}M_u(z,l_1l_2^2,\delta)\ll \(\Lambda+\Lambda\delta N^{1/2}y + \Lambda\delta^2y\)N^{\epsilon},
\end{align}
\begin{align}\label{mu3}
\sum_{1 \leqslant l_1l_2 \leqslant \Lambda}M_u(z,l_1^2l_2^2,\delta) \ll \(1+\delta N^{1/2}y + \delta^2y\)N^{\epsilon}.
\end{align}
\end{lem}

\begin{proof}
By \eqref{mainineq}, we need to count the number of matrices $\alpha = \left(\begin{array}{cc}a & b \\0 & d\end{array}\right)$ such that 
\begin{align}
|(a-d)z+b|\ll \delta y
\end{align}
for all the cases such that $ad = l_1$, $ad = l_1l_2$, $ad = l_1l_2^2$ and $ad = l_1^2l_2^2$.\\

We again consider the lattice $\braket{1,z}$ of covolume $y$ and shortest length at least $N^{-1/2}$ in $\mathbb{C}$. By \eqref{CP}, in each case, we have $\ll 1+\frac{\delta y}{N^{-1/2}} + \frac{\delta^2y^2}{y}$ possible values of $(a-d,b)$. In the first case, we have either $a=1$ or $d=1$ since $ad=l_1$, which gives rise of $O(1)$ possible matrices. In the next two cases, we have $O(\Lambda)$ possible values of $d$ because $ad=l_1l_2$ and $ad=l_1l_2^2$ respectively. In the last case, since both $l_1, l_2$ are primes, we have either $(a=1, d=l_1^2l_2^2)$ or $(a=l_1,d=l_1l_2^2)$ or $(a=l_1^2, d=l_2^2)$, or equivalent configurations. In each configuration, and for a given value $a-d$, there are $\ll N^{\epsilon}$ many pairs of $(a,d)$. Therefore, the proof is completed.
\end{proof}

\section{The Estimation of Parabolic Matrices}
In this section, we establish the upper bound of sum over parabolic matrices. The treatment in \cite{HT3} doesn't apply to this case, since $|u_{\alpha}(z)|^{-k}$ decays much slower than the geometric side of pre-trace formula in Maa{\ss} form case. We need a more careful discussion here.\\

%Let 
%$$\mathfrak{A}_{y_0} := \left\{z\in \mathbb{H} : \Im(z) > y_0\right\}.$$
%For a given $\alpha \in A_0(N)$, let 
%$$\alpha.\mathfrak{A}_{y_0}:= \left\{\alpha.z : z\in \mathfrak{A}_{y_0}\right\}.$$

%By Lemma \ref{A-Inv}, \cite{GHM} Propositions 2.4 and 2.5, when $z \in \alpha.\mathfrak{A}_{y_0}$, we have that 
%\begin{align}\label{trivial-sup}
%|y^{k/2}f(x+iy)| \ll N^{\epsilon}y_0^{-1/2}.
%\end{align}
Denote by $A_0(N) \backslash \mathbb{H}^2$ the fundamental domain of Atkin-Lehner operators.
%For the other $z$, 
%We then have the following estimation:
\begin{lem}\label{para}
Let $z \in A_0(N) \backslash \mathbb{H}^2$, $N^{-O(1)} \ll y \ll 1$ and $k \geqslant 2$, we have that
\begin{align}
\sum_{\substack{\alpha \in G_l(N)\\ \alpha \text{ parabolic}}} |u_{\alpha}(z)|^{-k} \ll_{\epsilon} \theta(l)2^{-k}l^{(-k+1)/2}\(y +N^{-1/3}y^{1/3}+N^{-5/3}y^{-4/3}+N^{-1}\)N^{\epsilon},
\end{align}
where $\theta(l) =1$ when $l$ is a perfect square and $\theta(l) =0$ otherwise. Furthermore, the implied constant does not depend on $k$.
\end{lem}

\begin{proof} When $l$ is not a square, there is no parabolic matrix by definition. Let $l$ be a square. Let $\alpha$ be an matrix in the sum. Since $\alpha$ is parabolic, there is a cusp $\mathfrak{a} \in P^1(\mathbb{Q})$ which is fixed by $\alpha$. Moreover, one can assume that $\mathfrak{a} =\frac{a}{c}$ for some $a,c \in \mathbb{Z}$. By the definition, when $a,c \neq 0$, we can assume that $(a,c)=1$. Let $\sigma_{\mathfrak{a}}$ be a 2-by-2 matrix such that $\sigma_{\mathfrak{a}}.\infty = \mathfrak{a}$ and
$$\sigma_{\mathfrak{a}} = \left(\begin{array}{cc}a & b \\ c & d\end{array}\right) \in SL_2(\mathbb{Z}),$$
for some $b,d\in \mathbb{Z}$.\\

Consider $\alpha\' = \sigma_{\mathfrak{a}}^{-1}\alpha\sigma_{\mathfrak{a}}$. We have that $\alpha \'.\infty =\infty$. This shows that $\alpha \'$ is an upper-triangular matrix. Since it is parabolic with determinant $l$, it must be of the form
$$\alpha \' = \pm\begin{pmatrix}\sqrt{l} & t \\0 & \sqrt{l}\end{pmatrix}.$$

For each $\alpha$, we have found an upper-triangular matrix $\alpha \' $ through the adjoint action of $\sigma_{\mathfrak{a}}$. Then we count the sum over $\alpha$s by parameterizing them as pairs $(\alpha \' , \sigma_{\mathfrak{a}})$.\\
%For such an $\alpha$, we can always find an upper-triangular matrix $\alpha \' $ through the adjoint action of $\sigma_{\mathfrak{a}}$. In order to estimate the sum over $\alpha$, we consider which $\alpha$ can be generated by a given upper-triangular matrix $\alpha \' $.\\

From the equation $\alpha = \sigma_{\mathfrak{a}}\alpha\' \sigma_{\mathfrak{a}}^{-1}$, we obtain that 
$$\alpha = \left(\begin{array}{cc}\sqrt{l}-act & a^2t \\-c^2t & \sqrt{l} +act \end{array}\right).$$
Since $\alpha \in G_l(N)$, we have $N|c^2t$. Furthermore, since $N$ is square-free, we have $r,s \in \mathbb{Z}$ such that $rs=N$, and $s|c$, $(c,r)=1$ and $r|t$.
%Let $z\' : = \sigma_{\mathfrak{a}}^{-1}z$. By brute force computation, we have 
%\begin{align}
%u_{\alpha}(z) = u_{ \sigma_{\mathfrak{a}}^{-1}\alpha\' \sigma_{\mathfrak{a}}}( \sigma_{\mathfrak{a}}.z\') = -i\sqrt{l}-\frac{t}{y\'},
%\end{align}
%where 
%$$y\' = \text{Im }z\' =  \text{Im }(\sigma_{\mathfrak{a}}^{-1}z) = \frac{y}{r\left|-sz+a\right|^2}.$$

When $t=0$, all the $\alpha = \pm\left(\begin{array}{cc}\sqrt{l} & 0 \\0 & \sqrt{l}\end{array}\right)$ are the same. 
When $t \neq 0$ and $c=0$, we set $a=1$. When $t \neq 0$ and $a=0$, we set $c=1$. Moreover, $\left|u_{\alpha}(z)\right| = \left|2\sqrt{l}yi+t|cz-a|^2\right|y^{-1}$.

Therefore, we have
\begin{align}
\sum_{\substack{\alpha \in G_l(N)\\ \alpha \text{ parabolic}}} |u_{\alpha}(z)|^{-k} &\ll 2^{-k}l^{-k/2}
+\sum_{t\neq 0}\frac{y^k}{\left|2\sqrt{l}yi+t\right|^k}+\sum_{N|t, t\neq 0}\frac{y^k}{\left|2\sqrt{l}yi+t|z|^2\right|^k}
+\sum_{\substack{a,c,t\neq 0\\ \text{s.t. }\alpha \in G_l(N)}}\frac{y^k}{\left|2\sqrt{l}yi+t|cz-a|^2\right|^k}\\
%&\ll l^{-k/2}+\sum_{rs=M}\sum_{t\neq 0}\frac{y^k}{\left|\sqrt{l}yi+rt\min_{a \in \mathbb{Z}}\{|a-sz|^2\}\right|^k}\\
& \ll 2^{-k}l^{-k/2}+\sum_{t\neq 0}\frac{y^k}{\(2\sqrt{l}y\)^{k\alpha}|t|^{k\beta}}+\sum_{N|t, t\neq 0}\frac{y^k}{\(2\sqrt{l}y\)^{k\alpha}|t|z|^2|^{k\beta}}
+\sum_{\substack{a,c,t\neq 0\\ \text{s.t. }\alpha \in G_l(N)}}\frac{y^k}{\(2\sqrt{l}y\)^{k\alpha}\(t|cz-a|^2\)^{k\beta}},\label{para_ineq}
% \ll l^{-k/2}+\sum_{\substack{t\neq 0, (c,a) \neq (0,0)\\ \text{s.t. }\alpha \in G_l(N)}}\frac{y^k}{\(\sqrt{l}y\)^{k\alpha}\(t|cz-a|^2\)^{k\beta}},\label{para_ineq}
\end{align}
by Arithmetic-Geometric Mean Inequality for some positive $\alpha, \beta$ such that $\alpha+\beta=1$. Moreover, the implied constant is absolute and independent of $k$.\\
 
Now let $k\beta = 1+\epsilon$ for some positive $\epsilon<\frac{1}{2}$. By noticing that $|z|^2\geqslant 1/N$ when $z$ is in the fundamental domain, the sum of first three terms is easy to obtain.
%The sum of first two terms is bounded
%$$\ll_{\epsilon} l^{-k/2} + l^{-(k-1)/2}y(yl)^{\epsilon}.$$ 
Let $t=rt_1$ and $c=sc_1$ in the fourth sum, then $(sc_1, ra)=1$ by the choices of $a,c,r,s$. Then \eqref{para_ineq} is bounded by 
\begin{align}\label{para_ineq2}
 %\ll  l^{-k/2} + l^{-(k-1)/2}y(yl)^{\epsilon}+\sum_{rs=N}\sum_{\substack{c_1, t_1\neq 0,a}}\frac{\(l^{\frac{1}{2}}y\)^{k\beta}}{l^{\frac{k}{2}}\(|t_1|r|sc_1z-a|^2\)^{k\beta}}
  \ll_{\epsilon}  2^{-k}\(l^{-k/2} + l^{-(k-1)/2}y(yl)^{\epsilon}+\sum_{rs=N}\sum_{\substack{c_1,a\\ (sc_1, ra)=1}}\frac{\(l^{\frac{1}{2}}y\)^{1+\epsilon}}{l^{\frac{k}{2}}\(r|sc_1z-a|^2\)^{1+\epsilon}}\).
\end{align}

Let $1 \leqslant R\leqslant N$. Break the $r,s$ sum apart as
 \begin{align}
\( \sum_{\substack{rs=N\\r > R}}+ \sum_{\substack{rs=N\\s \geqslant N/R}}\)\sum_{\substack{c_1,a\\(sc_1, ra)=1}}\frac{\(l^{\frac{1}{2}}y\)^{1+\epsilon}}{l^{\frac{k}{2}}\(r|sc_1z-a|^2\)^{1+\epsilon}}.
 \end{align}
 
First consider the case that $r > R$. Since $z$ is in the fundamental domain, there are integers $b\'$ and $d\'$ such that 
\begin{align}
\Im\(\left(\begin{array}{cc}\sqrt{r}a & b\'/\sqrt{r} \\\sqrt{r}sc_1 & \sqrt{r}d\'\end{array}\right).z\) = \frac{y}{r|sc_1z-a|^2}\leqslant y,
\end{align}
which implies that $r|sc_1z-a|^2 \geqslant 1$. Applying Lemmas \ref{Countlem1}, \ref{Countlem2} to lattice $\braket{1,z}$, we consider the
value of $|sc_1z-a|^2$ dyadically to obtain 
\begin{align}
\sum_{\substack{rs=N\\r > R}}\sum_{\substack{c_1,a\\ (sc_1, ra)=1}}\frac{\(l^{\frac{1}{2}}y\)^{1+\epsilon}}{l^{\frac{k}{2}}\(r|sc_1z-a|^2\)^{1+\epsilon}}\ll_{\epsilon} N^{\epsilon}\(1+\frac{N^{1/2}}{R^{1/2}}+\frac{1}{Ry}\)(l^{1/2}y)^{1+\epsilon}l^{-\frac{k}{2}}.
\end{align}

Next consider the case that $s \geqslant N/R$. We open the norm square to obtain
\begin{align}
\sum_{\substack{rs=N\\s \geqslant N/R}}\sum_{\substack{c_1,a\\ (sc_1,ar)=1}}\frac{\(l^{\frac{1}{2}}y\)^{1+\epsilon}}{l^{\frac{k}{2}}\(r|sc_1z-a|^2\)^{1+\epsilon}}&= \sum_{\substack{rs=N\\s \geqslant N/R}}\sum_{\substack{c_1,a\\ (sc_1,ar)=1}}\frac{\(l^{\frac{1}{2}}y\)^{1+\epsilon}}{l^{\frac{k}{2}}r^{1+\epsilon}\((sc_1x-a)^2+(sc_1y)^2\)^{1+\epsilon}}\\
& \ll  \sum_{\substack{rs=N\\s \geqslant N/R}}\(\sum_{\substack{|sc_1x-a|<1\\ (sc_1,ar)=1}}\frac{\(l^{\frac{1}{2}}y\)^{1+\epsilon}}{l^{\frac{k}{2}}r^{1+\epsilon}(sc_1y)^{2+2\epsilon}} + \sum_{\substack{|sc_1x-a|\geqslant 1\\ (sc_1,ar)=1}}\frac{\(l^{\frac{1}{2}}y\)^{1+\epsilon}}{l^{\frac{k}{2}}r^{1+\epsilon}\(|sc_1x-a|sc_1y\)^{1+\epsilon}}\)\\
&\ll_{\epsilon} N^{\epsilon}l^{-\frac{k}{2}}\(\(\frac{l^{\frac{1}{2}}R}{N^{2}y}\)^{1+\epsilon}+\(\frac{l^{\frac{1}{2}}}{N}\)^{1+\epsilon}\).
\end{align}

We then choose $R = N^{5/3}y^{4/3}$ to complete the proof.
%Introduce an elementary inequality such that when $T>0$, $\alpha,\beta>0$ and $\alpha+\beta>1$, we have
%$$\sum_{n=1}^{\infty}\frac{1}{n^{\alpha}(n^{\beta}+T)}\ll_{\alpha,\beta} T^{\frac{1-\alpha-\beta}{\beta}}\log T.$$
%Hence, when $y\ll 1$, \eqref{para_ineq} is 
%$$\ll_{k,\epsilon} \(l^{-k/2}+(My)^{-1}l^{(1-k)/2}+(M)^{-1}l^{(1-k)/2}\)\log \(ML\).$$
\end{proof}
%\begin{lem}
%When $1 \leqslant \delta < y^{-1}$,
%\begin{align}\label{para}
%M_p(z,l,\delta) \leqslant 2\theta(l)
%\end{align}
%where $\theta(l) = 1$ when $l$ is a perfect square and $0$ otherwise. Furthermore, the follwoing bound holds for any $\delta$.\\
%\end{lem}

\section{The Proof of Theorem \ref{sup}}
By \eqref{zregion}, it suffices to consider the case that $y \geqslant \frac{\sqrt{3}}{2N}$. By Proposition \ref{trivial-sup}, when $z \in A_0(N) \backslash \mathbb{H}^2$ and $\Im z > N^{-2/3}$ we have $\left|y^{k/2}f(z)\right| \ll k^{\frac{1}{4}+\epsilon}N^{-\frac{1}{6}+\epsilon}\braket{f, f}^{1/2}$. Thus, we only need to show the sup-norm when $z \in A_0(N) \backslash \mathbb{H}^2$ and $\frac{\sqrt{3}}{2}N^{-1} \leqslant \Im(z) \leqslant N^{-2/3}$.\\

In \eqref{supineq}, one has
%\begin{align}\label{meq}
%L^{2-\epsilon}\left|y^{k/2}f_i(z)\right|^2 \ll \sum_l |y_l|l^{\frac{k-1}{2}}\int^{\infty}_{C\sqrt{l}}\frac{M(z,l,\delta)}{\delta^{k+1}}d\delta = \sum_l |y_l|l^{\frac{k-1}{2}}\int^{\infty}_{C\sqrt{l}}\frac{\(M_u+M_*\)(z,l,\delta)}{\delta^{k+1}}d\delta\\
%\end{align}
%Moreover, we have 
\begin{align}
|y_l| \ll 
\begin{cases}
L, & l=1, \\ 1, & l=l_1 \text{ or }l_1l_2\text{ or }l_1l_2^2\text{ or }l_1^2l_2^2\text{ with }L<l_1,l_2<2L\text{ primes,}\\ 0, & \text{ otherwise}.
\end{cases}
\end{align}

Next, we consider the contribution of upper-triangular, parabolic and generic matrices separately on the right hand side of \eqref{supineq}. Since $\delta$ is always larger than $2\sqrt{l}$, all the $k$-aspect implied constant of the symbol $\ll$ below is $2^{-k}$.

\subsubsection{Upper-triangular}
When $l=1$, we choose $\Lambda = 1$ in \eqref{mu0}, then this part contributes $\ll$\\ $N^{\epsilon}L\(1+N^{1/2}y+y\)$. When $l = l_1$, via \eqref{mu0} again, then the upper bound is $ \ll N^{\epsilon}L^{-1/2}\(1+L^{1/2}N^{1/2}y+Ly\)$. When $l =l_1l_2$, via \eqref{mu1}, the upper bound is $\ll  N^{\epsilon}L^{-1}\(L+L^2N^{1/2}y+L^3y\)$. When $l=l_1l_2^2$, via \eqref{mu2} the upper bound is $\ll  N^{\epsilon}L^{-3/2}\(L+L^{5/2}N^{1/2}y+L^4y\)$. When $l=l_1^2l_2^2$, via \eqref{mu3} the upper bound is $\ll  N^{\epsilon}L^{-2}\(1+L^2N^{1/2}y+L^4y\)$. Therefore, the total contribution is $\ll N^{\epsilon}\(L+LN^{1/2}y+L^{5/2}y\)$. Notice that $k > 3$, so every integral is convergent.

 \subsubsection{Parabolic}
From Lemma \ref{para}, we know that when $l=1,l_1^2$, the upper bound is $$\ll L\(y+N^{-1/3}y^{1/3}+N^{-5/3}y^{-4/3}\)N^{\epsilon},$$ and when $l=l_1^2l_2^2$,  the upper bound is $$\ll L^{2}\(y+N^{-1/3}y^{1/3}+N^{-5/3}y^{-4/3}\)N^{\epsilon}.$$ When $l$ is not a square, there is no contribution from parabolic case. Hence the total contribution from generic case is $\ll  L^{2}\(y+N^{-1/3}y^{1/3}+N^{-5/3}y^{-4/3}\)N^{\epsilon}$.
 
 \subsubsection{Generic}
 When $l=1$, via \eqref{m*1}, the upper bound is $\ll N^{\epsilon}L\((Ny)^{-1}+N^{-1/2}+N^{-1}\)$. 
 When $l=l_1$, via \eqref{m*1}, the upper bound is $\ll N^{\epsilon}L^{-1/2}\(L(Ny)^{-1}+L^{3/2}N^{-1/2}+L^2N^{-1}\)$. When $l=l_1l_2$, via \eqref{m*1}, the upper bound is $\ll N^{\epsilon}L^{-1}\(L^2(Ny)^{-1}+L^{3}N^{-1/2}+L^4N^{-1}\)$.
 
  When $l=l_1l_2^2$, via \eqref{m*3}, the upper bound is $\ll N^{\epsilon}L^{-3/2}\(L^{3/2}(Ny)^{-1}+L^{3}N^{-1/2}+L^{9/2}N^{-1}\)$. When $l=l_1^2l_2^2$, via \eqref{m*2}, the upper bound is $\ll N^{\epsilon}L^{-2}\(L^2(Ny)^{-1}+L^{4}N^{-1/2}+L^6N^{-1}\)$. Hence the total contribution from generic case is $\ll  N^{\epsilon}\(L(Ny)^{-1} + L^2N^{-1/2} + L^4N^{-1}\)$. For the convergence, we need to use Lemma \ref{CON} when $\delta$ is sufficiently large.\\ 
 
Therefore, we choose $L=N^{1/3}$ in \eqref{supineq} to obtain
$$\frac{\left|y^{k/2}f_i(z)\right|^2}{\braket{f_i,f_i}} \ll kN^{-1/3+\epsilon},$$
which implies Lemma \ref{sup}.\\
%\begin{align}
%y^kh(z,-\overline{z}) = \sum_{\gamma \in \Gamma_0(N)}\frac{y^k}{(j(\gamma,z))^k(-\overline{z}+\gamma.z)^k}=  \sum_{\gamma \in \Gamma_0(N)} u_{\gamma}(z)^{-k}.
%\end{align}

%Moreover, we have
%\begin{align}
%|u_{\gamma}(z)|^2= \frac{|\overline{z}-\gamma.z|^2}{\text{Im}(\gamma.z)\text{Im}(-\overline{z})}
%\end{align}

\end{document}